\documentclass[a4paper,11pt,leqno,twoside]{article}
\usepackage{amsmath,amsthm,amssymb}
\usepackage[dvips]{graphicx,color,psfrag}
\usepackage{rotating}
\usepackage{multirow,makecell}

\numberwithin{equation}{section}

\newcommand{\cS}{\mathcal{S}}
\newcommand{\cL}{\mathcal{L}}
\newcommand{\cX}{\mathcal{X}}

\newcommand{\m}{\mu}

\newcommand{\RB}{\mathrm{RB}}
\newcommand{\Conj}{\mathrm{Conj}}

\newcommand{\Gauss}[1]{\lfloor{#1}\rfloor}

\newcommand{\boldtitle}[1]{\title{\bfseries #1}}

\newenvironment{MSC}{%
\smallbreak
\noindent \textbf{2010\ Mathematics Subject Classification\,:}}

\newenvironment{keywords}{%
\noindent\textbf{Key words and phrases\,:}\itshape}

\theoremstyle{theorem}

\theoremstyle{definition}
\newtheorem*{multiproclaim}{\variable@name}

\theoremstyle{plain}
\newtheorem{thm}{Theorem}[section]

\newtheorem{lem}[thm]{Lemma}
\newtheorem{cor}[thm]{Corollary}

\theoremstyle{definition}

\newtheorem{remark}[thm]{Remark}

\boldtitle{\bf Ramanujan Cayley graphs of Frobenius groups} 
\author{
 Miki HIRANO\thanks{Partially supported by Grant-in-Aid for Scientific Research (C) No. 24540022.},
 Kohei KATATA
 and
 Yoshinori YAMASAKI
\thanks{Partially supported by Grant-in-Aid for Young Scientists (B) No. 24740018.}
}
\date{\today}

\begin{document}

\setlength{\baselineskip}{15pt}
\maketitle 

\begin{abstract}
 In this paper,
 we determine the bound of the valency of Cayley graphs of Frobenius groups with respect to normal Cayley subsets   
 which guarantees to be Ramanujan.
 We see that if the ratio between the orders of the Frobenius kernel and complement is not so small, 
 then this bound coincides with the trivial one
 coming from the trivial estimate of the largest non-trivial eigenvalue of the graphs.
 Moreover, in the cases of the dihedral groups of order twice odd primes,
 which are special cases of the Frobenius groups,
 we determine the same bound for the Cayley graphs of the groups with respect to not only normal but also all Cayley subsets.
 As is the case of abelian groups which we have treated in the previous papers,
 such a bound is equal to the trivial one in the above sense or,
 as exceptional cases, exceeds one from it.
 We then clarify that the latter occurs if and only if
 the corresponding prime is represented by a quadratic polynomial in a finite family.
\begin{MSC}
 {\it Primary}
 05C50;
 {\it Secondary}
 05C25,
 11M41,
 11N32.
\end{MSC} 
\end{abstract}

\section{Introduction}

 Let $X$ be a $k$-regular graph with standard assumptions, that is, finite, undirected, connected and simple. 
 The graph $X$ is called {\it Ramanujan} if its largest non-trivial eigenvalue (in the sense of absolute value)
 is not greater than the Ramanujan bound $2\sqrt{k-1}$. 
 The Ramanujan property of the graph means that the associated Ihara zeta function satisfies the ``Riemann hypothesis'',
 which enables us to have a good  estimate for the number of the prime cycles in it (see, e.g., \cite{Terras2011}). 
 See \cite{Lubotzky2012} for the other relations between this property and various mathematical objects. 

 We have considered the following problem in our previous papers \cite{HiranoKatataYamasaki,Katata2014}. 
 Let $G$ be a finite group and $\cS$ a set of Cayley subsets of $G$ which includes $G\setminus\{1\}$
 with $1$ being the identity element of $G$. 
 Put
\[
 \cX
=\cX_{G,\cS}
=\left\{\left.X(S)\,\right|\,S\in\cS\right\},
\]
 where $X(S)$ is the Cayley graph of $G$ with respect to the Cayley subset $S\in\cS$. 
 Letting $\cL=\{|G|-|S|\,|\,S\in\cS\}$ be the set of ``covalencies'' of graphs in $\cX$, 
 we write $\cX=\bigsqcup_{l\in\cL}\cX_l$ where $\cX_l=\{X(S)\in\cX\,|\,|G|-|S|=l\}$. 
 Notice that $\cX_1=\{K_{|G|}\}$ where $K_{|G|}=X(G\setminus \{1\})$ is the complete graph with $|G|$-vertices. 
 According to \cite{AlonRoichman1994}, some neighbors of $K_{|G|}$ are expected to be Ramanujan.
 We want to estimate them precisely, that is, to determine the number 
\[
 \hat{l}_{G,\cS}
=\max\left\{\left.l\in \cL\,\right|\,\text{$X\in\cX_k$ is Ramanujan for all $1\le k\le l$}\right\} 
\]
 of edge-removal preserving the Ramanujan property from the complete graph $K_{|G|}\in\cX$.

 Previously, we have discussed this problem when $G$ is abelian and $\cS=\widetilde{\cS}$ is the set of {\it all} Cayley subsets of $G$. 
 In this paper, we treat the cases when $G$ is a Frobenius group
 and $\cS=\cS_0$ is consist of all {\it normal} Cayley subsets of $G$. 
 Here we call a Cayley subset normal if it is a union of conjugacy classes of $G$, 
 and a Frobenius group is a non-abelian finite group
 such as the dihedral group $D_{2m}$ of order $2m$ with odd $m$,
 the non-abelian $pq$-group $F_{p,q}$ with odd primes $p$, $q$, etc
 (for the definition of the Frobenius groups, see \cite{CurtisReiner1981}). 
 The main result is the following theorem. 

\begin{thm}
\label{thm:MainResult}
 Let $G=N\rtimes H$ be a Frobenius group
 where $N$ and $H$ are the Frobenius kernel and complements, respectively,
 and $\cS_0$ the set of all normal Cayley subsets of $G$. 
 If $\frac{|N|-1}{|H|}\ge 4$, then it holds that 
\[
 \hat{l}_{G,\cS_0}=l_{0},
\]
 where $l_{0}=\max\bigl\{l\in\cL\,\bigr|\,l\le 2(\sqrt{|G|}-1)\bigr\}$.
\end{thm}

 We remark that, when $G$ is non-abelian,
 the situation is completely different from whether $S$ is in $\cS_0$ or not.
 Actually, when $S$ is normal, we have expressions of eigenvalues of $X(S)$ in terms of irreducible characters of $G$. 
 Nevertheless, for the case of the dihedral group $G=D_{2m}$ of order $2m$ with odd $m$,
 one can calculate eigenvalues of $X(S)$ explicitly for any Cayley subset $S\in\widetilde{\cS}$
 because the dimensions of irreducible representations of $D_{2m}$ are at most two.
 As a consequence, one can show that there exists $\varepsilon_m\in \{0,1\}$ such that 
\[
 \hat{l}_{D_{2m},\widetilde{\cS}}=\hat{l}_{D_{2m},\cS_0}+\varepsilon_m
\] 
 for each $m$.
 For simplicity, we restrict to the case where $m=p$ is prime in this paper.
 When $\Gauss{2\sqrt{2p}}$ is even
 (here, for $x\in\mathbb{R}$, $\Gauss{x}$ represents for the largest integer not exceeding $x$),
 from Theorem~\ref{thm:MainResult},
 one immediately shows that $\varepsilon_p=0$.
 On the other hand when $\Gauss{2\sqrt{2p}}$ is odd,
 we can see that most of $\varepsilon_p$'s are equal to $0$. 
 Now let us call $p$ {\it exceptional} if $\Gauss{2\sqrt{2p}}$ is odd and $\varepsilon_p=1$ and {\it ordinary} otherwise. 
 Similarly to the abelian results, 
 we have the following theorem for $D_{2p}$. 

\begin{thm}
 The odd prime $p$ is exceptional if and only if it is of the form of one of the following quadratic polynomials;
\[
 p=
\begin{cases}
 2k^2+4k-3 & (k\ge 5),\\
 2k^2+4k-1 & (k\ge 3),\\
 2k^2+4k+1 & (k\ge 3)\\
\end{cases}
\quad \text{or} \quad 
 p=
\begin{cases}
 2k^2+6k-1 & (k\ge 7),\\
 2k^2+6k+1 & (k\ge 3),\\
 2k^2+6k+3 & (k\ge 3).
\end{cases}
\]
\end{thm}

 The classical Hardy-Littlewood conjecture asserts that
 every quadratic polynomials express infinitely many primes under some standard conditions.
 We finally remark that
 the above result implies that
 there exists infinitely many exceptional primes 
 if and only if the Hardy-Littlewood conjecture is true 
 for at least one of the above six quadratic polynomials.

 Throughout of the present paper,
 we denote by $\mathbb{P}$ the set of all odd primes and
 $\mathbb{Z}_m=\mathbb{Z}/m\mathbb{Z}$ for $m\in\mathbb{Z}_{\ge 1}$.
 
\section{Preliminary}

\subsection{Cayley graphs}

 Let $X$ be a $k$-regular graph with $m$-vertices ($m<\infty$) which is undirected, connected and simple. 
 The {\it adjacency matrix} $A_X$ {\it of} $X$ is the symmetric matrix of size $m$
 whose entry is $1$ if the corresponding pair of vertices are connected by an edge and $0$ otherwise. 
 We call the eigenvalues of $A_X$ the {\it eigenvalues of} $X$. 
 The set $\Lambda(X)$ of all eigenvalues of $X$ is given as 
\[
 \Lambda(X)=\bigl\{\lambda_i\,\bigl|\,k=\lambda_0>\lambda_1\ge\cdots\ge\lambda_{m-1}\ge -k\bigr\}.
\]
 Let $\m(X)$ be the largest non-trivial eigenvalue of $X$ in the sense of absolute value, that is, 
\[
 \m(X)=\max\bigl\{|\lambda|\,\bigl|\,\lambda\in \Lambda{(X)}, \ |\lambda|\ne k\bigr\}.
\]
 Then, $X$ is called {\it Ramanujan} if the inequality $\mu(X)\le 2\sqrt{k-1}$ holds. 
 Here the constant $2\sqrt{k-1}$ is often called the {\it Ramanujan bound for} $X$ and is denoted by $\RB(X)$. 

 Let $G$ be a finite group with the identity element $1$
 and $S$ a Cayley subset of $G$, that is, it is a symmetric generating subset of $G$ without $1$.
 We denote by $X(S)$ the {\it Cayley graph} of $G$ with respect to the Cayley subset $S$.
 This is undirected, connected and simple $|S|$-regular graph 
 with the vertex set $G$ and the edge set $\{(x,y)\in G^2\,|\,x^{-1}y\in S\}$. 
 In what follows, for a Cayley subset $S$,
 we write $\Lambda(S)=\Lambda(X(S))$, $\mu(S)=\mu(X(S))$, $\RB(S)=\RB(X(S))$, and so on.  
 It is well known that
 if the Cayley subset $S$ of $G$ is normal, that is, it is a union of conjugacy classes of $G$,
 then the eigenvalues of $X(S)$ can be written by irreducible characters of $G$.

\begin{lem}
[{\it cf.} \cite{KrebsShaheen2011}]
\label{lem:normalCayleyeigenvalues}
 Let $G$ be a finite group and $\widehat{G}$ the set of all irreducible characters of $G$.
 For a normal Cayley subset $S$ of $G$,
 we have $\Lambda(S)=\{\lambda_{\chi}\,|\,\chi\in \widehat{G}\}$ where 
\begin{equation}
\label{for:EigenvalueforNormalCay}
 \lambda_{\chi}
=\frac{1}{\chi(1)}\sum_{s\in S}\chi(s)
\end{equation} 
 with the multiplicity $\chi(1)^2$.
\qed
\end{lem}   

\subsection{A problem for Ramanujan Cayley graphs}

 Let $\cS$ be the set of all normal Cayley subsets of $G$.
 For $S\in\cS$, we define $l(S)=|G|-|S|=|G\setminus S|$ and call it a {\it covalency} of $X(S)$.
 Letting $\cL=\{l(S)\,|\,S\in\cS\}$, we write $\cS=\bigsqcup_{l\in\cL}\cS_l$ where $\cS_{l}=\{S\in\cS\,|\,l(S)=l\}$.
 As we have explained in the introduction, our purpose is to determine 
\[
 \hat{l}=\hat{l}(G)
=\max\left\{l\in \cL\,\left|\,\text{$X(S)$ is Ramanujan for all $S\in\cS_k$ $(1\le k\le l)$}\right.\right\}.
\]
 The following lemma is fundamental.
\begin{lem}
\label{lem:trivialestimate}
 Let $S\in\cS$.
 If $l(S)\le 2(\sqrt{|G|}-1)$, then $X(S)$ is Ramanujan.
\end{lem}
\begin{proof}
 Take $S\in\cS$ with $l(S)\le \frac{|G|}{2}$.
 For any non-trivial irreducible character $\chi\in\widehat{G}$, we have 
\[
 \lambda{\chi}
=\frac{1}{\chi(1)}\sum_{s\in S}\chi(s)
=-\frac{1}{\chi(1)}\sum_{s\in G\setminus S}\chi(s)
\]
 from the orthogonality.
 This together with $|\chi(s)|\le \chi(1)$ shows that
 $|\lambda_{\chi}|\le \min\{|S|,l(S)\}=l(S)$.
 Hence, if $l(S)\le\RB(S)=2\sqrt{|G|-l(S)-1}$,
 equivalently $l(S)\le 2(\sqrt{|G|}-1)$, then $X(S)$ is Ramanujan.
 Remark that $2(\sqrt{|G|}-1)\le \frac{|G|}{2}$ for any $G$.
\end{proof}

 Now, let
\[
 l_{0}
=l_0(G)
=\max\left\{l\in\cL\,\left|\,l\le 2(\sqrt{|G|}-1)\right.\right\}.
\]
 From Lemma~\ref{lem:trivialestimate}, we have $ l_0\le \hat{l}$.
 We call $l_0$ a trivial (lower) bound of $\hat{l}$.

\section{Ramanujan Frobenius graphs}

 In this section, we study Cayley graphs of Frobenius groups with respect to normal Cayley subsets. 
 We call such graphs {\it Frobenius graphs}.
 From now on, for a group $G$ and $x,y\in G$, we write $x^{y}=y^{-1}xy$.
 Moreover, we denote by $\Conj_{G}(x)$ the conjugacy class of $x\in G$ in $G$ and
 by $c(G)$ the number of conjugacy classes in $G$.

\subsection{Character table of the Frobenius group}

 Let $G=N\rtimes H$ be a Frobenius group,
 where $N$ and $H$ are subgroups of $G$ called the Frobenius kernel and complement, respectively
 (see \cite{CurtisReiner1981} for details of the Frobenius groups and their characters).
 Notice that  
\[
 r=\frac{|N|-1}{|H|}
\]
 is a positive integer.
 
 We first recall the character table of the Frobenius group $G$.
 It is known that a set of all representatives of the conjugacy classes of $G$ can be taken as 
 $\{1\}\sqcup \{x_i\}^{k}_{i=1}\sqcup\{y_j\}^{h}_{j=1}$ where   
 $x_i\in N$ ($1\le i\le k$) and $y_j\in H$ ($1\le j\le h$) with $k=\frac{c(N)-1}{|H|}$ and $h=c(H)-1$.
 Notice that $\Conj_{G}(x_i)\subset N$ and $\Conj_{G}(y_j)\subset G\setminus N$ with 
\[
 |\Conj_{G}(x_i)|=|\Conj_{N}(x_i)||H|, \quad 
 |\Conj_{G}(y_j)|=|\Conj_{H}(y_i)||N|.
\]
 The irreducible characters of $G$ are given as follows.
 Since $H\simeq G/N$, a non-trivial irreducible character of $H$
 is corresponding to that of $G$ which has the kernel containing $N$.
 We write these as $\chi_{\alpha}$ ($1\le \alpha\le h$).
 Moreover, for a non-trivial irreducible character $\psi_{\beta}$ of $N$,
 its induced character is again an irreducible character of $G$.
 We write these as $\phi_{\beta}=\mathrm{Ind}(\psi_{\beta})$ ($1\le \beta\le k$).
 Notice that 
\[
 \phi_{\beta}(x)
=\frac{1}{|N|}\sum_{y\in G \atop x^y\in N}\psi_{\beta}(x^y)
=\sum_{z\in H}\psi_{\beta}(x^z).
\] 
 It is known that these together with the trivial character ${\bf 1}$ exhaust all irreducible characters of $G$. 
 Now the character table of $G$ is given as follows:

\begin{table}[!h]
\begin{center}
\renewcommand{\arraystretch}{1.3}
\begin{tabular}{c||c|c|c}
  & $1$ & $x_i \ (1\le i\le k)$ & $y_j \ (1\le j\le h)$ \\
\hline
\hline
 ${\bf 1}$ & $1$ & $1$ & $1$ \\
\hline
 $\chi_{\alpha} \ (1\le \alpha\le h) $ & $\chi_{\alpha}(1)$ & $\chi_{\alpha}(1)$ & $\chi_{\alpha}(y_j)$ \\
\hline
 $\phi_{\beta} \ (1\le \beta\le k)$ & $|H|\psi_{\beta}(1)$ & $\displaystyle{\sum_{z\in H}\psi_{\beta}(x^z)}$ & $0$
\end{tabular}
\caption{The character table of the Frobenius group $G=N\rtimes H$.}
\end{center}
\end{table}

\subsection{Eigenvalues of Frobenius graphs}

 Now, let us calculate the eigenvalues of Frobenius graphs with respect to normal Cayley subsets. 

 For subsets $X\subset \{x_i\}^{k}_{i=1}$ and $Y\subset \{y_j\}^{h}_{j=1}$,
 we put $S_{X,Y}=S_X\sqcup S_Y$ where 
\[
 S_{X}=\bigsqcup_{x\in X}\Conj_G(x), \quad
 S_Y=\bigsqcup_{y\in Y}\Conj_G(y).
\] 
 We say that $X\subset \{x_i\}^{k}_{i=1}$ (resp. $Y\subset \{y_j\}^{h}_{j=1}$) is {\it symmetric}
 if $S_X$ (resp. $S_Y$) is symmetric.
 It is clear that $S_{X,Y}$ is symmetric if and only if both $X$ and $Y$ are. 
 We have  
\[
 \cS\subset 
\left\{S_{X,Y}\,\left|\,\text{both $X\subset \{x_i\}^{k}_{i=1}$ and $Y\subset\{y_j\}^{h}_{j=1}$ are symmetric}\right.\right\}.
\]
 Notice that, if $S_{X,Y}$ in the right hand side satisfies $|S_{X,Y}|>\frac{|G|}{2}$,
 then it generates $G$ and hence is in $\cS$ and, moreover, 
 $Y\ne \emptyset$ because otherwise $|S_{X,Y}|=|S_{X,\emptyset}|<|N|-1$, which is contradict to $\frac{|G|}{2}\ge |N|$. 
 This means that, in the determination of $l_0$ and $\hat{l}$,
 we may assume that $S\in\cS$ is always of the form of $S=S_{X,Y}$ for some symmetric subsets $X\subset \{x_i\}^{k}_{i=1}$ and $\emptyset\ne Y\subset\{y_j\}^{h}_{j=1}$.
 
 From Lemma~\ref{lem:normalCayleyeigenvalues}, we have 
 
\begin{lem}
\label{lem:specG}
 For $X\subset\{x_i\}^{k}_{i=1}$ and $\emptyset\ne Y\subset\{y_j\}^{h}_{j=1}$, we have
\[
 \Lambda(S_{X,Y})
=\{\lambda_{\bf 1}\}\cup\{\lambda_{\chi_{\alpha}}\}^{h}_{\alpha=1}\cup \{\lambda_{\phi_{\beta}}\}^{k}_{\beta=1},
\]
 where
\begin{align}
\label{for:SXY}
 \lambda_{\bf 1}
&=|S_{X,Y}|=|H|\sum_{x\in X}|\Conj_N(x)|+|N|\sum_{y\in Y}|\Conj_H(y)|,\\
 \lambda_{\chi_{\alpha}}
&=|H|\sum_{x\in X}|\Conj_{N}(x)|+\frac{|N|}{\chi_{\alpha}(1)}\sum_{y\in Y}\chi_{\alpha}(y)|\Conj_{H}(y)|,\nonumber\\
 \lambda_{\phi_{\beta}}
&=\frac{1}{\psi_{\beta}(1)}\sum_{x\in X}\sum_{z\in H}\psi_{\beta}(x^{z})|\Conj_{N}(x)|\nonumber
\end{align}
 with the multiplicities $1$, $\chi_{\alpha}(1)^2$ and $|H|^2\psi_{\beta}(1)^2$, respectively.
\end{lem}
\begin{proof}
 These follow from the character table of $G$ obtained in the previous subsection.
\end{proof}

\subsection{Main results}

 To determine $l_0$ and $\hat{l}$,
 we at first describe the set $\cL=\{\l(S)\,|\,S\in\cS\}$.

 For symmetric subsets $X\subset\{x_i\}^{k}_{i=1}$ and $\emptyset\ne Y\subset\{y_j\}^{h}_{j=1}$,
 we respectively put 
\[
 a_X=r-\sum_{x\in X}|\Conj_N(x)|, \quad 
 b_Y=|H|-1-\sum_{y\in Y}|\Conj_H(y)|.
\]
 We have $0\le a_X\le r$, $0\le b_{Y}<|H|-1$ and from \eqref{for:SXY}
\begin{equation}
\label{for:lSXY}
 l(S_{X,Y})=|G|-|S_{X,Y}|=1+a_{X}|H|+b_Y|N|.
\end{equation} 
 
\begin{lem}
\label{lem:cardinarity}
 Let $X,X'\subset\{x_i\}^{k}_{i=1}$ and $\emptyset\ne Y,Y'\subset\{y_j\}^{h}_{j=1}$ be symmetric subsets.
 Then, $l(S_{X,Y})=l(S_{X',Y'})$ if and only if $(a_X,b_Y)=(a_{X'},b_{Y'})$.
\end{lem}
\begin{proof}
 From \eqref{for:lSXY}, 
 one sees that $l(S_{X,Y})=l(S_{X',Y'})$ is equivalent to $(a_{X}-a_{X'})|H|+(b_{Y}-b_{Y'})|N|=0$.
 Since $(a_{X}-a_{X'})|H|<|N|$, it has to hold that $a_X=a_{X'}$ and hence $b_{Y}=b_{Y'}$.
\end{proof}

 Put $l(a,b)=1+a|H|+b|N|$.
 From Lemma~\ref{lem:cardinarity},
 we write $\cS=\bigsqcup_{a\in A,b\in B}\cS_{l(a,b)}$ where 
\begin{align*}
 A&=\left\{a_{X}\,\left|\,\text{$X\subset\{x_i\}^{k}_{i=1}$ is symmetric}\right.\right\}=\{a_1<a_2<\cdots<a_{m}\},\\
 B&=\left\{b_{Y}\,\left|\,\text{$\emptyset\ne Y\subset\{y_j\}^{h}_{j=1}$ is symmetric}\right.\right\}=\{b_1<b_2<\cdots<b_n\},
\end{align*}
 with $m=|A|$ and $n=|B|$.
 Remark that $a_1=0$ and $a_{m}=r$, which respectively correspond to the cases $X=\{x_i\}^{k}_{i=1}$ and $X=\emptyset$.
 Similarly,  $b_1=0$, which corresponds to the case $Y=\{y_j\}^{h}_{j=1}$, and $b_{n}<|H|-1$ because $Y\ne\emptyset$.
 Moreover, when $h\ge 2$, since the center of $H$ is not trivial,
 there exists $y'\in \{y_j\}^{h}_{j=1}$ such that $\Conj_H(y')=\{y'\}$.
 This implies that $b_2=1$ if $\{y'\}$ is not symmetric, that is, $y'^2\ne 1$ and $2$ otherwise.
 We below describe the relations among $l(a,b)$ for $a\in A$ and $b\in B$; 
\begin{align*}
 1&=l(a_1,b_1)<l(a_2,b_1)<\cdots<l(a_m,b_1)=(b_1+1)|N|=|N|\\
<b_2|N|+1&=l(a_1,b_2)<l(a_2,b_2)<\cdots<l(a_m,b_2)=(b_2+1)|N|\\
<b_3|N|+1&=l(a_1,b_3)<l(a_2,b_3)<\cdots<l(a_m,b_3)=(b_3+1)|N|\\
& \,\ \vdots\\
<b_n|N|+1&=l(a_1,b_n)<l(a_2,b_n)<\cdots<l(a_m,b_n)=(b_n+1)|N|.
\end{align*}

 The followings are our main results,
 which are involved to the determinations of $l_0$ and $\hat{l}$
 for Frobenius graphs with respect to normal Cayley subsets.  
 
\begin{thm}
\label{thm:main}
 Assume that $r=\frac{|N|-1}{|H|}\ge 4$.
\begin{itemize}
\item[$(1)$] There exists $1\le i_0 <m$ such that $l_0=l(a_{i_0},b_1)<|N|$.
\item[$(2)$] It holds that $\hat{l}=l_0$.
\end{itemize}
\end{thm}
\begin{proof}
 Under the condition $r\ge 4$,
 we have $|H|\le \frac{|N|-1}{4}<\frac{|N|}{4}$
 and hence $2(\sqrt{|G|}-1)<2\sqrt{|N||H|}<2\sqrt{|N|\frac{|N|}{4}}=|N|$.
 Therefore, the first assertion follows from the definition of $l_0$.

 To prove the second one, it is sufficient to show that 
 there exists $S\in\cS_{l(a_{i_0+1},b_1)}$ such that $X(S)$ is not Ramanujan.
 Actually, take any $S=S_{X,Y}\in \cS_{l(a_{i_0+1},b_1)}$.
 Then, since $Y=\{y_j\}^{h}_{j=1}$, we have from Lemma~\ref{lem:specG}
\begin{align*}
 \lambda_{\chi_{\alpha}}(S_{X,Y})
&=|H|\sum_{x\in X}|\Conj_{N}(x)|+\frac{|N|}{\chi_{\alpha}(1)}\sum^{h}_{i=1}\chi_{\alpha}(y_j)|\Conj_{H}(y_j)|\\
&=|H|\Bigl(\frac{|N|-1}{|H|}-a_{i_0+1}\Bigr)+\frac{|N|}{\chi_{\alpha}(1)}(-\chi_{\alpha}(1))\\
&=-(1+a_{i_0+1}|H|)\\
&=-l(a_{i_0+1},b_1)\\
&=-l(S_{X,Y}).
\end{align*}
 Here, the second equality follows from the orthogonality of characters 
 together with the fact that $\chi_{\alpha}$ is regarded as a non-trivial irreducible character of $H$.  
 This implies that $|\lambda_{\chi_{\alpha}}(S_{X,Y})|=l(S_{X,Y})>l_0$
 and hence $|\lambda_{\chi_{\alpha}}(S_{X,Y})|>\RB(S)$ by the definition of $l_0$.
 This ends the proof. 
\end{proof}

 We remark that,
 since $l(a_{i_0},b_1)\le 2(\sqrt{|G|}-1)<l(a_{i_0+1},b_1)$ with $l(a_i,b_1)=1+a_i|H|$,
 it can be expressed as
\begin{equation}
\label{for:a0}
 a_{i_0}=\max\Bigl\{a\in A\,\Bigl|\,a\le\frac{2\sqrt{|N||H|}-3}{|H|}\Bigr\}.
\end{equation}

 Let us calculate $\hat{l}$ for the following typical examples of the Frobenius groups.

 First, for $p\in\mathbb{P}$, 
 consider the dihedral group   
\[
 D_{2p}=\mathbb{Z}_{p}\rtimes \mathbb{Z}_{2}=\langle x,y\,|\,x^{p}=y^{2}=1,\ y^{-1}xy=x^{-1}\rangle,
\]
 which is a Frobenius group with $r=\frac{p-1}{2}$.
 In this case,
 one can take representatives of the conjugacy classes of $D_{2p}$ as $\{1\}\sqcup \{x^{v}\}^{\frac{p-1}{2}}_{v=1}\sqcup \{y\}$.
 Since all the conjugacy classes 
 $\Conj_{D_{2p}}(x^v)=\{x^{v},x^{-v}\}$ for $1\le v\le \frac{p-1}{2}$ and $\Conj_{D_{2p}}(y)=\{x^ay\,|\,0\le a\le p-1\}$
 are symmetric, 
 we have $A=\{i\,|\,0\le i\le \frac{p-1}{2}\}$ and $B=\{0\}$.
 Now, from Theorem~\ref{thm:main} together with \eqref{for:a0},
 we obtain the following result (notice that $r\ge 4$ if $p\ge 11$). 
 
\begin{cor}
\label{cor:D2p}
 Let $G=D_{2p}$ where $p\in\mathbb{P}$.
 If $p\ge 11$, then we have 
\begin{equation*}
 \hat{l}=l_0
=2\Gauss{\sqrt{2p}-\frac{1}{2}}-1.
\end{equation*}
 \qed 
\end{cor}
 
 We next consider the group  
\[
 F_{p,q}=\mathbb{Z}_{p}\rtimes_{u} \mathbb{Z}_{q}=\langle x,y\,|\,x^{p}=y^{q}=1,\ y^{-1}xy=x^u\rangle,
\]
 where $p,q\in\mathbb{P}$ with $q\,|\,(p-1)$ and $u$ an element of $\mathbb{Z}^{\times}_p$ of order $q$.
 It is known that $F_{p,q}$ is also a Frobenius group with $r=\frac{p-1}{q}$.
 Let $S=\langle u \rangle=\{1,u,u^2,\ldots,u^{q-1}\}$ and
 $V=\{v_1,\ldots,v_r\}$ a set of all representatives of $\mathbb{Z}^{\times}_p/S$.
 Then, one can take representatives of the conjugacy classes of $F_{p,q}$ as
 $\{1\}\sqcup \{x^{v}\}_{v\in V}\sqcup \{y^{b}\}^{q-1}_{b=1}$.
 Because none the conjugacy classes 
 $\Conj_{F_{p,q}}(x^v)=\{x^{vs}\,|\,s\in S\}$ for $v\in V$ and
 $\Conj_{F_{p,q}}(y^b)=\{x^ay^b\,|\,0\le a\le p-1\}$ for $1\le b\le q-1$ are symmetric, 
 we have $A=\{2i\,|\,0\le i\le \frac{p-1}{2q}\}$ and $B=\{2j\,|\,0\le j< \frac{q-1}{2}\}$.
 Similarly as above, one obtains the following result (notice that $r\ge 4$ if $p\ge 4q+1$).

\begin{cor}
\label{cor:Fpq}
 Let $G=F_{p,q}$ where $p,q\in\mathbb{P}$ with $q\,|\,(p-1)$.
 If $p\ge 4q+1$, then we have
\begin{equation*}
\hat{l}=l_0=2q\Gauss{\frac{2\sqrt{pq}-3}{2q}}+1.
\end{equation*}
 \qed
\end{cor}

\begin{remark}
 There are several examples of the Frobenius groups with $r=\frac{|N|-1}{|H|}\le 3$
 for which the claims in Theorem~\ref{thm:main} do not hold.
 Actually, consider the dihedral groups $D_{2p}$ for $p=3,5,7$,
 which are corresponding to the cases $r=1,2,3$, respectively.
 In these cases, we have $(m,n)=(\frac{p+1}{2},1)$ and can check that $l_0=l(a_{\frac{p-1}{2}},b_1)=p-2$.
 Moreover, it holds that $\hat{l}=l(a_{\frac{p+1}{2}},b_1)=p$
 because the corresponding Cayley graph is $X(S_{\emptyset,\{y\}})$,
 which is Ramanujan because $\Lambda(S_{\emptyset,\{y\}})=\{\pm p,0\}$.
\end{remark}

\section{Ramanujan dihedral graphs}

 In this section,
 we more precisely study Ramanujan Cayley graphs of the dihedral groups $D_{2p}$ of order $2p$ with $p\in\mathbb{P}$,
 which are special cases of the Frobenius groups.
 For simplicity, we call such graphs {\it dihedral graphs}.
 
\subsection{An universal problem on Ramanujan Cayley graphs}

 We here consider our problem in more general situation.
 For a finite group $G$, let $\widetilde{\cS}$ be the set of {\it all} Cayley subsets of $G$.
 Similarly to the case of $\cS$, 
 we put $\widetilde{\cL}=\{l(S)\,|\,S\in \widetilde{\cS}\}$,
 and write $\widetilde{\cS}=\bigsqcup_{l\in\widetilde{\cL}}\widetilde{\cS}_l$
 where $\widetilde{\cS}_{l}=\{S\in\widetilde{\cS}\,|\,l(S)=l\}$.
 Now our interest is whether one can determine
\[
 \tilde{l}
=\tilde{l}(G)
=\max\left\{l\in \widetilde{\cL}\,\left|\,\text{$X(S)$ is Ramanujan for all $S\in\widetilde{\cS}_k$ ($1\le k\le l$)}\right.\right\}.
\] 
 We remark that when $G$ is abelian, which we have studied in \cite{HiranoKatataYamasaki,Katata2014},
 $\tilde{l}=\hat{l}$ because $\tilde{\cS}=\cS$.

 The determination of $\tilde{l}$ is much more difficult rather than that of $\hat{l}$
 since one does not have explicit expressions of eigenvalues, such as \eqref{for:EigenvalueforNormalCay}, of general Cayley graphs.
 However, in the case of the dihedral group,
 we can manage this problem because of the fact that 
 the dimensions of irreducible representations are at most two
 together with the relation 
\begin{equation}
\label{def:l1}
 \tilde{l}<l_1=\min\{\l\in\cL\,|\,l>l_0\},
\end{equation}
 which follows directly from the definition of $\hat{l}$ and Theorem~\ref{thm:main}.
 
\subsection{Initial results}

 Let us consider the dihedral graph $X(S)$ of $D_{2p}$ with respect to $S\in\widetilde{\cS}$.
 Divide $D_{2p}$ into two parts as $D_{2p}=D_{1}\sqcup D_{2}$ 
 where $D_{1}=\{1,x,x^2,\ldots,x^{p-1}\}$ and $D_{2}=\{y,xy,x^2y,\ldots,x^{p-1}y\}$.
 According to this decomposition, 
 we write $S=S_{1}\sqcup S_{2}$ and $l(S)=l_1(S)+l_2(S)$ where $S_{i}=S\cap D_{i}$ and $l_{i}(S)=|D_{i}\setminus S_i|=p-|S_{i}|$ for $i=1,2$.
 Remark that, since any subset of $D_2$ is symmetric because the order of any element in $D_2$ is two, 
 $S$ is symmetric if and only if $S_{1}$ is, which implies that $|S_{1}|$ is always even and hence $l_1(S)$ is odd.
 Define $z_j=z_j(S),w_j=w_j(S)\in\mathbb{C}$ ($0\le j\le p-1$) by 
 $z_{0}=|S_{1}|+|S_{2}|$, $w_0=|S_{1}|-|S_{2}|$ and
\begin{equation}
\label{def:zw}
 z_{j}=\sum_{x^a\in S_{1}}\omega^{ja}=-\sum_{x^a\in D_1\setminus S_{1}}\omega^{ja}, \qquad
 w_{j}=\sum_{x^ay \in S_{2}}\omega^{ja}=-\sum_{x^ay \in D_2\setminus S_{2}}\omega^{ja}
\end{equation}
 for $1\le j\le p-1$.
 Here, $\omega=e^{\frac{2\pi i}{p}}$.
 Note that $z_{j}\in\mathbb{R}$ because $S_{1}$ is symmetric. 
 It is known that the eigenvalues of $X(S)$ are described by using $z_j$ and $w_j$.

\begin{lem}
\begin{itemize}
\item[$\mathrm{(i)}$] $\Lambda(S)=\{\mu^{(+)}_{j},\mu^{(-)}_{j}\,|\,0\le j\le p-1\}$ where $\mu^{(\pm)}_{j}=z_j\pm |w_j|$.
\item[$\mathrm{(ii)}$] Let $|\mu_j|=\max\{|\mu^{(+)}_{j}|,|\mu^{(-)}_{j}|\}$.
 Then, we have $|\mu_j|=|z_{j}|+|w_{j}|$.
\end{itemize}
\end{lem}
\begin{proof}
 See, e.g., \cite{Babai1979} for the first assertion.
 The second one is direct.
\end{proof}

 One can obtain a lower bound of $\tilde{l}$
 coming from the trivial estimate of the eigenvalues.
 
\begin{lem}
\label{lem:trivialbound}
 For $p\ge 29$, we have $\tilde{l}\ge \Gauss{2\sqrt{2p}}-2$.  
\end{lem}
\begin{proof}
 We first remark that $\mu^{(+)}_{0}=|S_{1}|+|S_{2}|=|S|$ is the largest eigenvalue of $X(S)$
 and hence can write $\mu(S)=\max\bigl\{|\mu|\,|\,\mu\in \Lambda(S),\ |\mu|\ne \mu^{(+)}_0\bigr\}=\max\{|\mu^{(-)}_0|,|\mu_1|,\ldots,|\mu_{p-1}|\}$.

 Assume that $l(S)\le \frac{p}{2}$, which implies that $l_i(S)\le \frac{p}{2}$ for $i=1,2$.
 Then, from \eqref{def:zw}, for $1\le j\le p-1$, we see that  
 $|\mu_j|=|z_{j}|+|w_{j}|\le \min\bigl\{|S_{1}|,l_1(S)\bigr\}+\min\bigl\{|S_{2}|,l_2(S)\bigr\}\le l_1(S)+l_2(S)=l(S)$.
 Moreover, it holds that $\mu^{(-)}_0=2|S_1|-2p+l(S)\le l(S)$.
 These show that $\mu(S)\le l(S)$. 
 Therefore, if $l(S)\le \mathrm{RB}(S)=2\sqrt{2p-l(S)-1}$,
 equivalently $l(S)\le\Gauss{2\sqrt{2p}}-2$,
 then $X(S)$ is Ramanujan.
 Notice that $2\sqrt{2p}-2\le \frac{p}{2}$ when $p\ge 29$.
\end{proof}

 By virtue of results on the Frobenius graphs obtained in the previous section,  
 one gets an upper bound of $\tilde{l}$.
 As a consequence, we can narrow the candidates of $\tilde{l}$ down to at most two.   

\begin{thm}
\label{thm:epsilon}
 Assume that $p\ge 29$.
\begin{itemize}
\item[$\mathrm{(i)}$] If $\Gauss{2\sqrt{2p}}$ is even, then we have $\tilde{l}=\hat{l}+1$.  
\item[$\mathrm{(ii)}$] If $\Gauss{2\sqrt{2p}}$ is odd, then we have $\tilde{l}=\hat{l}$ or $\tilde{l}=\hat{l}+1$.  
\end{itemize} 
 Here $\hat{l}=2\Gauss{\sqrt{2p}-\frac{1}{2}}-1$ is obtained in Corollary~\ref{cor:D2p}.
\end{thm}
\begin{proof}
 We first remark that, for $\alpha\in\mathbb{R}$, it holds that 
\[
 2\Gauss{\alpha-\frac{1}{2}}-1
=
\begin{cases}
 \Gauss{2\alpha}-2-1 & (\text{$0\le \alpha-\Gauss{\alpha}<\frac{1}{2}$ or, $\Gauss{2\alpha}$ is even}), \\
 \Gauss{2\alpha}-2 & (\text{$\frac{1}{2}\le \alpha-\Gauss{\alpha}<1$ or, $\Gauss{2\alpha}$ is odd}).
\end{cases} 
\] 
 Using this formula with $\alpha=\sqrt{2p}$,
 we see that $\Gauss{2\sqrt{2p}}-2$ coincides with $\hat{l}+1$ (resp. $\hat{l}$)
 if $\Gauss{2\sqrt{2p}}$ is even (resp. odd) and hence, from Lemma~\ref{lem:trivialbound},
 $\tilde{l}\ge \hat{l}+1$ (resp. $\tilde{l}\ge \hat{l}$).
 Now one obtains the results because $\tilde{l}<l_1=\hat{l}+2$,
 which follows from \eqref{def:l1}.
\end{proof}

 From this theorem,
 it can be written as $\tilde{l}=\hat{l}+\varepsilon$ for some $\varepsilon\in\{0,1\}$.
 As is the case of the circulant graphs \cite{HiranoKatataYamasaki,Katata2014},
 we call $p$ {\it exceptional} if $\Gauss{2\sqrt{2p}}$ is odd and $\varepsilon=1$ and {\it ordinary} otherwise.
 Now our task is to clarify which $p\in\mathbb{P}$ is exceptional and whether such primes exist infinitely many.


\begin{remark}
 The discussion in this section can be extended to that for $D_{2m}$ where $m$ is odd (not necessary prime),
 as we have done in the case of the circulant graphs in \cite{HiranoKatataYamasaki,Katata2014}.
 However, for simplicity, we only show results in the case where $m$ is odd prime. 
\end{remark}

\subsection{A characterization of exceptional primes}

 In what follows, we assume that $\Gauss{2\sqrt{2p}}$ is odd.
 
 For $l\in\widetilde{\cL}$,
 let $\mu(l)=\max\{\mu(S)\,|\,S\in\widetilde{\cS}_{l}\}$ and $\RB(l)=\RB(S)=2\sqrt{2p-l-1}$ for $S\in \widetilde{\cS}_{l}$.
 From the definition, $p$ is exceptional if and only if $\mu(\hat{l}+1)\le \RB(\hat{l}+1)$.
 To study this inequality, we at first construct $S\in\widetilde{\cS}_{\hat{l}+1}$ such that $\mu(\hat{l}+1)=\mu(S)$.

 For $l\in\widetilde{\cL}$,
 let $L(l)=\{(l_1,l_2)\in\mathbb{Z}^2_{\ge 0}\,|\,\text{$l_1+l_2=l$,\ $l_1$ is odd}\}$.
 Moreover, for $(l_1,l_2)\in L(l)$, define $S^{(l_1,l_2)}=S^{(l_1)}_1\sqcup S^{(l_2)}_2\in \mathcal{S}_{l}$ by
 $S^{(l_1)}_1=D_{1}\setminus\{1,x^{\pm 1},x^{\pm 2},\ldots,x^{\pm\frac{l_1-1}{2}}\}$ and 
 $S^{(l_2)}_2=D_{2}\setminus\{y,xy,x^2y.\ldots,x^{l_2-1}y\}$.
 One sees that $l_i(S^{(l_1,l_2)})=l_i$ for $i=1,2$ and   
\[
 z_j
=\sum^{\frac{l_1-1}{2}}_{h=-\frac{l_1-1}{2}}\omega^{hj}
=\frac{\sin{\frac{\pi jl_1}{p}}}{\sin{\frac{\pi j}{p}}}, \qquad 
 w_j
=\sum^{l_2-1}_{h=0}\omega^{hj}
=\omega^{\frac{j(l_2-1)}{2}}\frac{\sin{\frac{\pi jl_2}{p}}}{\sin{\frac{\pi j}{p}}},
\]
 whence $|\mu_j|=|\mu_j(l_1,l_2)|$ can be written as 
\[
 |\mu_j|
=|z_j|+|w_j|
=\frac{\sin{\frac{\pi jl_1}{p}}}{\sin{\frac{\pi j}{p}}}+\frac{\sin{\frac{\pi jl_2}{p}}}{\sin{\frac{\pi j}{p}}}
=2\frac{\sin{\frac{\pi jl}{2p}}}{\sin{\frac{\pi j}{p}}}\cos{\frac{\pi j|l_1-l_2|}{2p}}.
\]

 Now let us write $\hat{l}=2\Gauss{\sqrt{2p}-\frac{1}{2}}-1=\Gauss{2\sqrt{2p}}-2$ as     
\[
 \hat{l}=4k+r
\]
 for some $k\ge 0$ and $r\in\{1,3\}$. 

\begin{lem}
\label{lem:maximum}
 Let $(\check{l}_1,\check{l}_2)=\bigl(\frac{\hat{l}+1}{2},\frac{\hat{l}+1}{2}\bigr)$ if $r=1$ and
 $\bigl(\frac{\hat{l}+3}{2},\frac{\hat{l}-1}{2}\bigr)$ otherwise.
 Then, we have
\[
 \mu(\hat{l}+1)
=\mu\bigl(S^{(\check{l}_1,\check{l}_2)}\bigl)=\bigl|\mu_1(\check{l}_1,\check{l}_2)\bigr|.
\]
\end{lem}
\begin{proof}
 Consider when $|l_1-l_2|$ takes minimum under the condition that $l_1+l_2=\hat{l}+1$ and $l_1$ is odd. 
 Note that, since $p$ is prime, $|\mu_1(\check{l}_1,\check{l}_2)|$ takes maximum
 among $|\mu_j(\check{l}_1,\check{l}_2)|$ for $1\le j\le p-1$. 
\end{proof}

 When $\hat{l}=4k+r$, we see that $p\in I_{r,k}\cap \mathbb{P}$ where 
\begin{align*}
 I_{r,k}
&=\bigl\{t\in\mathbb{R}\,\bigr|\,\Gauss{2\sqrt{2t}}-2=4k+r\bigr\}\\
&=\Bigl[2k^2+(r+2)k+\frac{(r+2)^2}{8},2k^2+(r+3)k+\frac{(r+3)^2}{8}\Bigr).
\end{align*}
 This means that
 it is expressed as $p=f_{r,c_r}(k)$ for some integers $k\ge 3$ and $c_r\in\mathbb{Z}$ 
 with $-k+2\le c_1\le 1$ if $r=1$ and $-k+4\le c_3\le 4$ otherwise.
 Here,
\[
 f_{r,c_r}(t)=2t^2+(r+3)t+c_r.
\] 
 Let $I_{r}=\bigsqcup_{k\ge 3}I_{r,k}\cap\mathbb{P}$ and $C_r=\{r-4,r-2,r\}$.
 Moreover, for $c_r\in C_r$, define an integer $k_{r,c_r}\ge 3$ by
 $(k_{1,-3},k_{1,-1},k_{1,1})=(5,3,3)$ and $(k_{3,-1},k_{3,1},k_{3,3})=(7,3,3)$.
 We now obtain the following theorem,
 which gives a characterization for exceptional primes.

\begin{thm}
\label{thm:mainD}
 A prime $p\in I_r$ with $p\ge 29$ is exceptional if and only if
 it is of the form of $p=f_{r,c_r}(k)$ for some $c_r\in C_r$ and $k\ge k_{r,c_r}$.
\end{thm}
\begin{proof}
 To clarify when the inequality $\mu(\hat{l}+1)=|\mu_{1}(\check{l}_1,\check{l}_2)|\le \RB(\hat{l}+1)$ holds, 
 we introduce an interpolation function $F_r(t)$ of the difference between  
 $|\mu_{1}(\check{l}_1,\check{l}_2)|$ and $\RB(\hat{l}+1)$ on $I_{r,k}$, that is,
\begin{align*}
 F_{r}(t)
&=2\frac{\sin{\frac{\pi (4k+r+1)}{2t}}}{\sin{\frac{\pi}{t}}}\cos{\frac{\pi(r-1)}{2t}}-2\sqrt{2t-4k-r-2}.
\end{align*}
 One can see that $F_{r}(t)$ is monotone decreasing on $I_{r,k}$ for sufficiently large $k$.
 Moreover, at $t=p=f_{r,c_r}(k)\in I_{r,k}\cap\mathbb{P}$, one has
\[
 F_{r}(p)=\frac{3(r+3)^2-24c_r-16\pi^2}{24}k^{-1}+O(k^{-2})
\]
 as $k\to\infty$ because
\begin{align*}
 |\mu_{1}(\check{l_1},\check{l_2})|
&=
 2\frac{\sin{\frac{\pi (4k+r+1)}{2(2k^2+(r+3)k+c_r)}}}{\sin{\frac{\pi}{2k^2+(r+3)k+c_r}}}
\cos{\frac{\pi(r-2)}{2(2k^2+(r+3)k+c_r)}}\\
&=4k+r+1-\frac{2\pi^2}{3}\frac{1}{k}+O(k^{-2}),\\
 \mathrm{RB}(\hat{l}+1)
&=2\sqrt{2(2k^2+(r+3)k+c_r)-4k-r-2}\\
&=4k+r+1-\frac{(r+3)^2-8c_r}{8}\frac{1}{k}+O(k^{-2}).
\end{align*}
 This shows that $F_{r}(p)<0$ for sufficiently large $k$ if and only if
 $3(r+3)^2-24c_r-16\pi^2<0$, that is, $c_r\in C_r$
 (note that $c_r$ should be odd because $p$ is).
 Actually, for each $r\in\{1,3\}$ and $c_r\in C_r$,
 we can check that
 the inequality $F_{r}(p)<0$ with $p=f_{r,c_r}(k)$ holds if and only if $k\ge k_{r,c_r}$.
 This completes the proof of the theorem.
\end{proof}

 For $r\in\{1,3\}$ and $c_r\in C_r$, let $J_{r,c_r}=\{f_{r,c_r}(k)\in I_r\,|\,k\ge k_{r,c_r}\}$.
 Namely, $J_{r,c_r}$ is the set of exceptional primes $p$ of the form of $p=f_{r,c_r}(k)$.
 These are given as follows:
\begin{align*}
 J_{1,-3}&=\{67,93,157,283,643,877,1453,3037,4603,5197,\ldots\}, \\
 J_{1,-1}&=\{29,47,197,239,389,509,719,797,2309,2447,\ldots\}, \\
 J_{1,1}&=\{31,71,97,127,199,241,337,449,577,647,\ldots\}, \\
 J_{3,-1}&=\{139,307,359,607,919,1399,1619,1979,2239,2659,\ldots\}, \\
 J_{3,1}&=\{37,109,541,757,1009,1297,1621,2377,6841,7561,\ldots\}, \\
 J_{3,3}&=\{59,83,179,263,311,419,479,683,839,1103,\ldots,\}.
\end{align*} 

 The classical Hardy-Littlewood conjecture \cite{HardyLittlewood1923} asserts that
 if a quadratic polynomial $f(t)=at^2+bt+c$ with $a,b,c\in\mathbb{Z}$ satisfies the conditions that
 $a>0$, $(a,b,c)=1$, $a+b$ and $c$ are not both even
 and $D=b^2-4ac$ is not a square, then there are infinitely many primes represented by $f(t)$ and, 
 moreover, that $\pi(f;x)=\#\{k\le x\,|\,f(k)\in\mathbb{P}\}$ obeys the asymptotic behavior 
\begin{equation}
\label{for:HL}
 \pi(f;x)\sim \frac{C(f)}{2}\frac{x}{\log{x}}, \qquad 
 C(f)=2\prod_{p\ge 3}\Bigl(1-\frac{\bigl(\frac{D}{p}\bigr)}{p-1}\Bigr),
\end{equation}
 as $x\to\infty$.
 Here, $C(f)$ is called the Hardy-Littlewood constant with $\bigl(\frac{D}{p})$ being the Legendre symbol.
 From Theorem~\ref{thm:mainD},
 one sees that the existence of infinitely many exceptional primes is related to this conjecture.

\begin{cor}
 There exists infinitely many exceptional primes if and only if 
 the Hardy-Littlewood conjecture is true for at least one of $f_{r,c_r}(t)$ for $r\in\{1,3\}$ and $c_r\in C_r$.
 \qed
\end{cor}

\begin{remark}
 From \eqref{for:HL}, we can expect that
 $\pi(f_{r,c_r};x)\sim \frac{C(f_{r,c_r})}{2}\frac{x}{\log{x}}$
 where 
\begin{align*}
 \frac{C(f_{r,c_r})}{2}
&=\prod_{p\ge 3}\Bigl(1-\frac{\bigl(\frac{c_r'}{p}\bigr)}{p-1}\Bigr)
=
\begin{cases}
\begin{cases}
 0.671043\ldots & (r=1,\ c_1=-3),\\
 1.03566\ldots & (r=1,\ c_1=-1),\\
 1.84998\ldots & (r=1,\ c_1=1),\\
\end{cases}
\\[20pt]
\begin{cases}
 1.14801\ldots & (r=3,\ c_3=-1),\\
 0.757353\ldots & (r=3,\ c_3=1),\\
 1.38332\ldots & (r=3,\ c_3=3),\\
\end{cases}
\end{cases}
\end{align*}
 with $c_1'=4-2c_1$ and $c_3'=9-2c_3$.
\end{remark}

\begin{remark}
 The existence of infinitely many {\it ordinary} primes are verified as follows.
 Let $J=\bigsqcup_{r\in\{1,3\}}\bigsqcup_{c_r\in C_r}J_{r,c_r}$. 
 Moreover, for a positive integer $a$,  
 let $J_{r,c_r}(a)=\{n\in\mathbb{Z}_{\ge 0}\,|\,0\le n\le a-1 \ \text{and $n\equiv f_{r,c_r}(k)\pmod{a}$ for some $0\le k\le a-1$}\}$
 and $J(a)=\bigcup_{r\in\{1,3\}}\bigcup_{c_r\in C_r}J_{r,c_r}(a)$.
 If we can take $b\in\{0,1,2,\ldots,a-1\}\setminus J(a)$ satisfying $(a,b)=1$,
 then we have $\{at+b\,|\,t\in\mathbb{Z}\}\cap J=\emptyset$
 and, from the Dirichlet theorem of arithmetic progression, can find infinitely many primes in $\{at+b\,|\,t\in\mathbb{Z}\}$.
 Now, these are ordinary from Theorem~\ref{thm:mainD}.
 To achieve such a purpose, one may take, for example, $(a,b)=(29,4),(35,8)$ or $(40,33)$. 
\end{remark}


\bigskip 

\noindent
\textsc{Miki HIRANO}\\
 Graduate School of Science and Engineering, Ehime University,\\
 Bunkyo-cho, Matsuyama, 790-8577 JAPAN.\\
 \texttt{hirano@math.sci.ehime-u.ac.jp}\\

\noindent
\textsc{Kohei KATATA}\\
 Graduate School of Science and Engineering, Ehime University,\\
 Bunkyo-cho, Matsuyama, 790-8577 JAPAN.\\
 \texttt{katata@math.sci.ehime-u.ac.jp}\\

\noindent
\textsc{Yoshinori YAMASAKI}\\
 Graduate School of Science and Engineering, Ehime University,\\
 Bunkyo-cho, Matsuyama, 790-8577 JAPAN.\\
 \texttt{yamasaki@math.sci.ehime-u.ac.jp}
 
\end{document}